\documentclass[11pt,twoside,leqno]{article}
\usepackage{amsmath,amsthm,amssymb}
\numberwithin{equation}{section}
\usepackage[all]{xy}
\usepackage[dvips]{graphicx,color,psfrag}
\usepackage{here}

\title{\bf Ramanujan Cayley graphs of the generalized quaternion groups and the Hardy-Littlewood conjecture}
\author{Yoshinori YAMASAKI\thanks{Partially supported by Grant-in-Aid for Scientific Research (C) No. 15K04785.}
} 
\date{\today}

\theoremstyle{theorem}

\theoremstyle{definition}
\newtheorem*{multiproclaim}{\variable@name}

\theoremstyle{plain}
\newtheorem{theorem}{Theorem}[section]
\newtheorem{proposition}[theorem]{Proposition}
\newtheorem{lemma}[theorem]{Lemma}
\newtheorem{corollary}[theorem]{Corollary}

\theoremstyle{definition}

\newenvironment{MSC}{%
\smallbreak
\noindent \textbf{2010\ Mathematics Subject Classification\,:}}

\newenvironment{keywords}{%
\noindent\textbf{Key words and phrases\,:}\itshape}

\newcommand{\cS}{\mathcal{S}}
\newcommand{\cL}{\mathcal{L}}

\newcommand{\RB}{\mathrm{RB}}
\newcommand{\Irr}{\mathrm{Irr}}
\newcommand{\Spec}{\mathrm{Spec}}

\newcommand{\Gauss}[1]{\lfloor{#1}\rfloor}

\begin{document}

\setlength{\baselineskip}{15pt}
\maketitle

\begin{abstract}
 In this article, 
 we investigate the bound of the valency of the Cayley graphs of the generalized quaternion groups
 which guarantees to be Ramanujan.
 As is the cases of the cyclic and dihedral groups in our previous studies, 
 we show that the determination of the bound in a special setting
 is related to the classical Hardy-Littlewood conjecture for primes represented by a quadratic polynomial. 
\begin{MSC}
 {\it Primary}
 11M41,
 {\it Secondary}
 05C25,
 05C75,
 11N32.
\end{MSC} 
\end{abstract}


\section{Introduction}
\label{sec:1}

 Expander graph is a sparse graph having strong connectivity properties. 
 Because of its rich theory with many applications,
 it is widely studied in various fields of mathematics such as
 combinatorics, group theory, differential geometry and number theory
 (see \cite{HooryLinialWigderson2006,Lubotzky2012} for survey of the expander graphs). 
 In particular, Ramanujan graph, which is an optimal expander graph in the sense of Alon Boppana's theorem
 and was first defined in \cite{{LubotzkyPhillipsSarnak1988}},
 plays an important role in not only pure mathematics but also applied mathematics. 
 Actually, because a graph is Ramanujan if and only if
 the associated Ihara zeta function satisfies the ``Riemann hypothesis'',
 it has a special interest for number theorists, especially who study zeta functions (see, e.g., \cite{Terras2011}). 
 Moreover, from the fact that a random walk on a Ramanujan graph quickly converges to the uniform distribution,
 it is used to construct a cryptography hash function \cite{CharlesLauterGoren2009}.
 From these reasons, it is worth finding or constructing Ramanujan graphs as many as possible,
 however, it is in general difficult.

 In this paper, we consider the following problem on Ramanujan graphs.
 Naively, one easily imagines that, because a Ramanujan graph has a strong connectivity property, 
 if we have a Ramanujan graph, then there expects to be another Ramanujan graph around it ({\it cf}. \cite{AlonRoichman1994}).
 This means that, even if we get rid of some edges from the given Ramanujan graph anyhow, it may remain to be Ramanujan.
 Now our problem is to clarify how many edges we can freely remove from the given Ramanujan graph with remaining to be Ramanujan in a given family of graphs.
 In particular, as a first stage, we consider this problem starting from the trivial Ramanujan graph, that is, the complete graph,
 in a family of Cayley graphs of a fixed group.
 Notice that, in this setting, removing edges corresponds to reducing elements of a Cayley subset of the group.
 See the end of Section~\ref{sec:Preliminary} for more precise mathematical formulation of our problem.

 In \cite{HiranoKatataYamasaki1},
 we first investigated this problem for the cyclic groups.
 Moreover, in \cite{HiranoKatataYamasaki2},
 we studied it for the dihedral groups, which are non-abelian (simplest) extensions of the previous case 
 (we actually consider this problem for groups in the class of the Frobenius groups in \cite{HiranoKatataYamasaki2},
 which contains for example the semi-direct product of the cyclic groups and hence, especially, the dihedral groups).
 In both cases, we showed that the determination of the above maximal number of removable edges
 (it corresponds to $\tilde{l}$ in our formulation) is related to the classical Hardy-Littlewood conjecture on analytic number theory, 
 which asserts that every quadratic polynomials express infinitely many primes under some standard conditions,
 if the order of the group is odd prime (resp. twice odd prime) in the case of the cyclic group (resp. the dihedral group).
 In succession to these cases, in the present paper,
 we work the same problem for the generalized quaternion group $Q_{4m}$
 and actually obtain the similar result (Theorem~\ref{thm:mainD}) 
 if we choose the set of Cayley graphs suitably.
 Notice that we indeed consider a wider class of groups
 in the sense that $Q_{4m}$ can not be expressed as any semi-directed product of the cyclic groups.
 We also remark that our discussion may be applied to groups
 whose maximal degree of the irreducible representations is at most two. 

 We use the following notations in this paper.
 The set of all real numbers, integers and odd primes are denoted by $\mathbb{R}$, $\mathbb{Z}$ and $\mathbb{P}$, respectively.
 For $x\in\mathbb{R}$, $\Gauss{x}$ (resp. $\lceil x\rceil$) denote the largest (resp. smallest) integer less (resp. greater) than or equal to $x$.
 We also note that the most of our numerical computations are performed by using Mathematica. 


\section{Preliminary}
\label{sec:Preliminary}

 In this section, we prepare some definitions and notations of graph theory,
 which are necessary for our discussion (see, more precisely, \cite{KrebsShaheen2011}). 
 Throughout this paper, all graphs are assumed to be finite, undirected, connected, simple and regular.
 
\smallbreak 

 Let $X$ be a $k$-regular graph with $m$-vertices. 
 The adjacency matrix $A_X$ of $X$ is the symmetric matrix of size $m$
 whose entry is $1$ if the corresponding pair of vertices are connected by an edge and $0$ otherwise. 
 We call the eigenvalues of $A_X$ the eigenvalues of $X$. 
 Let $\Lambda(X)$ be the set of all eigenvalues of $X$.
 We know that it can be written as  
 $\Lambda(X)=\bigl\{k=\lambda_0>\lambda_1\ge\cdots\ge\lambda_{m-1}\bigr\}\subset [-k,k]$.
 Let $\lambda(X)$ be the largest non-trivial eigenvalue of $X$ in the sense of absolute value;  
 $\lambda(X)=\max\bigl\{|\lambda|\,\bigl|\,\lambda\in \Lambda{(X)}, \ |\lambda|\ne k\bigr\}$.
 Then, $X$ is called Ramanujan if the inequality $\lambda(X)\le 2\sqrt{k-1}$ holds. 
 Here the constant $2\sqrt{k-1}$ is called the Ramanujan bound for $X$ and is denoted by $\RB(X)$. 

 Let $G$ be a finite group with the identity element $1$.
 Let $S$ be a Cayley subset of $G$, that is, $S$ is a symmetric generating subset of $G$ without $1$.
 We denote by $X(S)$ the Cayley graph of $G$ with respect to the Cayley subset $S$.
 This is $|S|$-regular graph whose vertex set is $G$ and edge set $\{(x,y)\in G^2\,|\,x^{-1}y\in S\}$. 
 Let $\cS_G$ be the set of all Cayley subset of $G$. 
 In what follows, for $S\in\cS$,
 we write $\Lambda(S)=\Lambda(X(S))$, $\lambda(S)=\lambda(X(S))$, $\RB(S)=\RB(X(S))$, and so on.  
 It is well known that the eigenvalues of $X(S)$ can be described in terms of the irreducible representations of $G$ as follows.

\begin{lemma}
[{\it cf}. \cite{Babai1979}]
\label{lem:eighnCayley}
 Let $G$ be a finite group and $\Irr(G)$ the set of all equivalence classes of the irreducible representations of $G$.
 Then, for $S\in\cS_G$, we have 
\[
 \Lambda(S)
=\bigcup_{\pi\in\Irr(G)}d_{\pi}\cdot \Spec(M_{\pi}(S)),
\]
 where, for $\pi\in \Irr(G)$, $d_{\pi}$ is the degree of $\pi$,
 $M_{\pi}(S)=\sum_{s\in S}\pi(s)$ and $\Spec(M_{\pi}(S))$ is the set of all eigenvalues of $M_{\pi}(S)$.
 Here, we understand that an element in $\Spec(M_{\pi}(S))$ is counted $d_{\pi}$ times in $d_{\pi}\cdot \Spec(M_{\pi}(S))$.
\qed
\end{lemma}

 We here explain our problem on Ramanujan graphs. 
 For a set $\cS\subset \cS_G$ of Cayley subsets of $G$,
 let $\cL=\cL_{G,\cS}=\{l(S)\,|\,S\in\cS\}$ where $l(S)=|G\setminus S|=|G|-|S|$ is the covalency of $S\in\cS$.
 Then, we have the decomposition $\cS=\sqcup_{l\in\cL}\cS_l$ with $\cS_l=\{S\in\cS\,|\,l(S)=l\}$.
 Now our aim is to determine the bound 
\[
 \tilde{l}=\tilde{l}_{G,\cS}
=\max\{l\in\cL\,|\,\text{$X(S)$ is Ramanujan for all $S\in \cS_k$ ($1\le k\le l$)}\}.
\]
 Remark that $\tilde{l}\ge 1$ if $G\setminus\{1\}\in\cS_1$ 
 because $X(G\setminus\{1\})$ is the complete graph $K_{|G|}$ with $|G|$ vertices, which is a (trivial) Ramanujan graph.
 Hence, in this case, roughly speaking, 
 $\tilde{l}$ represents the maximal number of removable edges from the complete graph $K_{|G|}$
 keeping to be Ramanujan.

 In this paper,
 we investigate $\tilde{l}$ when $G$ is the generalized quaternion group.
 

\section{Cayley graphs of the generalized quaternion groups}

 For a positive integer $m$, 
 the generalized quaternion group $Q_{4m}$ is defined by
\[
 Q_{4m}
=\left\langle x,y\,\left|\,x^{2m}=1,\ x^m=y^2,\ y^{-1}xy=x^{-1} \right.\right\rangle.
\]
 This is non-commutative unless $m=1$ and
 can not be expressed as a semi-direct product of any pair of subgroups of $Q_{4m}$.  
 One easily sees that the order of $Q_{4m}$ is $4m$
 because it has the expression 
\begin{align*}
 Q_{4m}
=\{x^ky^l\,|\,0\le k\le 2m-1, \ l=0,1\}
=\langle x\rangle\sqcup \langle x\rangle y,
\end{align*}
 where $\langle x\rangle=\{x^k\,|\,0\le k\le 2m-1\}$ and $\langle x\rangle y=\{x^ky\,|\,0\le k\le 2m-1\}$.
 Notice that $(x^k)^{-1}=x^{2m-k}$ and $(x^ky)^{-1}=x^{m+k}y$.
 
 To calculate the eigenvalues of the Cayley graph of $Q_{4m}$,
 we need the information about the conjugacy classes and the irreducible representations of $Q_{4m}$. 
 For $z\in Q_{4m}$, let $C(z)$ be the conjugacy class of $Q_{4m}$ containing $z$.
 Then, the following exhausts all conjugacy classes of $Q_{4m}$; 
 $C(1)=\{1\}$,
 $C(x^{k})=\{x^{k},x^{2m-k}\}$ ($1\le k\le m-1$),  
 $C(x^{m})=\{x^m\}$, 
 $C(y)=\{x^{2k}y\,|\,0\le k\le m-1\}$
 and $C(xy)=\{x^{2k+1}y\,|\,0\le k\le m-1\}$.
 Moreover, the irreducible representations of $Q_{4m}$ are given as follows;
 $\chi_1={\bf 1}$ (the trivial character), $\chi_2$, $\chi_3$ and $\chi_4$ which are of degree $1$ and 
 $\varphi_j$ ($1\le j\le m-1$) of degree $2$.
 We give the values of these representations in Table~1.
 Here, $\omega=e^{\frac{2\pi i}{2m}}$.

\begin{table}[htbp]
\begin{center}
{\small 
\begin{tabular}{ccc}
\renewcommand{\arraystretch}{1.3}
\begin{tabular}{c||c|c}
  & $x^k$ & $x^ky$ \\
\hline
\hline
 $\chi_1$ & $1$ & $1$  \\
\hline
 $\chi_{2}$ & $1$ & $-1$ \\
\hline
 $\chi_{3}$ & $(-1)^k$ & $i(-1)^k$ \\
\hline
 $\chi_{4}$ & $(-1)^{k}$ & $i(-1)^{k+1}$ \\
\hline
 $\varphi_{j}$ & $\left[\begin{array}{cc}\omega^{jk}&0\\0&\omega^{-jk}\end{array}\right]$ & $\left[\begin{array}{cc}0&\omega^{jk}\\(-1)^j\omega^{-jk}&0\end{array}\right]$
\end{tabular}
&
\qquad
&
\renewcommand{\arraystretch}{1.3}
\begin{tabular}{c||c|c}
  & $x^k$ & $x^ky$ \\
\hline
\hline
 $\chi_1$ & $1$ & $1$  \\
\hline
 $\chi_{2}$ & $1$ & $-1$ \\
\hline
 $\chi_{3}$ & $(-1)^k$ & $(-1)^k$ \\
\hline
 $\chi_{4}$ & $(-1)^{k}$ & $(-1)^{k+1}$ \\
\hline
 $\varphi_{j}$ & $\left[\begin{array}{cc}\omega^{jk}&0\\0&\omega^{-jk}\end{array}\right]$ & $\left[\begin{array}{cc}0&\omega^{jk}\\(-1)^j\omega^{-jk}&0\end{array}\right]$
\end{tabular}
\end{tabular}
}
\caption{The tables of the values of the irreducible representations of $Q_{4m}$:
 the left one is the case of odd $m$ and the right one is of even $m$.}
\end{center}
\end{table}

 From now on, we let $\cS=\cS_{Q_{4m}}$ be the set of all Cayley subsets of $Q_{4m}$.
 Let us calculate the eigenvalues of the Cayley graph $X(S)$ for $S\in\cS$.
 Put $S_1=S\cap \langle x\rangle$ and $S_2=S\cap \langle x\rangle y$ so that we can write   
\[
 S=S_1\sqcup S_2.
\]
 Moreover, put $l_1(S)=2m-|S_1|$ and $l_2(S)=2m-|S_2|$ so that $l(S)=l_1(S)+l_2(S)$.
 Notice that $S_1\ne \langle x\rangle$ since $1\notin S$ and hence $l_1(S)>0$ and
 $S_2\ne \emptyset$ because $S$ generates $Q_{4m}$ and therefore $l_2(S)<2m$.
 One sees that, because $S$ is symmetric, both $S_1$ and $S_2$ are also symmetric.
 This implies that they are respectively expressed as  
\begin{equation}
\label{for:Cayleysubset}
\begin{array}{rl}
 S_1
&=\displaystyle{\bigsqcup_{x^{k_1}\in S\atop 1\le k_1\le m-1}}\left\{x^{k_1},x^{2m-k_1}\right\} \sqcup \left\{x^{m}\right\}^{\delta},\\[5pt]
 S_2
&=\displaystyle{\bigsqcup_{x^{k_2}y\in S\atop 0\le k_2\le m-1}}\left\{x^{k_2}y,x^{m+k_2}y\right\},
\end{array}
\end{equation}
 where $\delta=\delta(S)=1$ if $x^m\in S$ and $0$ otherwise. 
 Here, we understand that $A^0=\emptyset$ and $A^1=A$ for any set $A$.
 From these expressions, we have  
 $l_2(S)\equiv 0$ \!\!\!\! $\pmod{2}$ and $l_1(S)\equiv l(S)\equiv \delta$ \!\!\!\! $\pmod{2}$.
 Based on the above expression,
 we obtain the following decomposition of $\cS$;
\begin{equation}
 \cS
=\bigsqcup_{l\in\cL}\cS_l
=\bigsqcup_{l\in\cL}\bigsqcup_{(l_1,l_2)\in\cL_{l}}\cS_{l_1,l_2},
\end{equation} 
 where 
\[
 \cL_{l}
=\left\{(l_1,l_2)\in\mathbb{Z}^2\,\left|\,
\begin{array}{l}
 0<l_1\le 2m, \ l_1\equiv l \!\!\! \pmod{2}, \\
 0\le l_2<2m, \ l_2\equiv 0 \!\!\! \pmod{2},
\end{array}
 \ l_1+l_2=l
\right.
\right\}
\]
 and $\cS_{l_1,l_2}=\{S\in\cS\,|\,l_1(S)=l_1,\,l_2(S)=l_2\}$ for $l\in\cL$ and $(l_1,l_2)\in\cL_l$.
  
 Put 
\begin{align*}
 \sigma^{\mathrm{e}}_{1}=\sigma^{\mathrm{e}}_{1}(S)
&=\#\{k_1\in\mathbb{Z}\,|\,\text{$1\le k_1\le m-1$, $x^{k_1}\in S$, $k_1\equiv 0$ \!\!\!\!\!\! $\pmod{2}$}\},\\
 \sigma^{\mathrm{o}}_{1}=\sigma^{\mathrm{o}}_{1}(S)
&=\#\{k_1\in\mathbb{Z}\,|\,\text{$1\le k_1\le m-1$, $x^{k_1}\in S$, $k_1\equiv 1$ \!\!\!\!\!\! $\pmod{2}$}\},\\
 \sigma^{\mathrm{e}}_{2}=\sigma^{\mathrm{e}}_{2}(S)
&=\#\{k_2\in\mathbb{Z}\,|\,\text{$0\le k_2\le m-1$, $x^{k_2}y\in S$, $k_2\equiv 0$ \!\!\!\!\!\! $\pmod{2}$}\},\\
 \sigma^{\mathrm{o}}_{2}=\sigma^{\mathrm{o}}_{2}(S)
&=\#\{k_2\in\mathbb{Z}\,|\,\text{$0\le k_2\le m-1$, $x^{k_2}y\in S$, $k_2\equiv 1$ \!\!\!\!\!\! $\pmod{2}$}\}.
\end{align*}
 Using these notations, it can be written as 
 $|S_1|=2(\sigma^{\mathrm{e}}_{1}+\sigma^{\mathrm{o}}_{1})+\delta$ and $|S_2|=2(\sigma^{\mathrm{e}}_{2}+\sigma^{\mathrm{0}}_{2})$.
 Note that 
\[
 0\le \sigma^{\mathrm{e}}_{1}\le\frac{m-1}{2}, \ \  
 0\le \sigma^{\mathrm{o}}_{1}\le\frac{m-1}{2}, \ \ 
 0\le \sigma^{\mathrm{e}}_{2}\le\frac{m+1}{2}, \ \ 
 0\le \sigma^{\mathrm{o}}_{2}\le\frac{m-1}{2}
\]
 if $m$ is odd and
\[
 0\le \sigma^{\mathrm{e}}_{1}\le\frac{m}{2}-1, \ \  
 0\le \sigma^{\mathrm{o}}_{1}\le\frac{m}{2}, \ \ 
 0\le \sigma^{\mathrm{e}}_{2}\le\frac{m}{2}, \ \
 0\le \sigma^{\mathrm{o}}_{2}\le\frac{m}{2}
\]
 otherwise.

 From Lemma~\ref{lem:eighnCayley} together with the expression \eqref{for:Cayleysubset} of the Cayley subset,
 one can explicitly calculate the eigenvalues of $X(S)$ as follows.
 
\begin{lemma}
\label{lem:eigenQ4m}
 For $S\in\cS$,
 we have  
\begin{align*}
 \Lambda(S)
&=1\cdot \{\lambda_{1},\lambda_{2},\lambda_{3},\lambda_{4}\}\cup
\bigcup^{m-1}_{j=1}2\cdot \left\{\mu_j^{+},\mu_j^{-}\right\}.
\end{align*}
 Here, for $1\le j\le m-1$, $\mu^{\pm}_{j}=\mu^{\pm}_{j}(S)$ is given by
 $\mu^{\pm}_{j}=z_j\pm |w_j|$ with 
\begin{align*}
 z_{j}=z_j(S)
&=\sum_{x^{k_1}\in S\atop 0\le k_1\le 2m-1}\omega^{jk_1}
=\sum_{x^{k_1}\in S\atop 1\le k_1\le m-1}(\omega^{jk_1}+\omega^{-jk_1})+\delta (-1)^j, \\
 w_{j}=w_j(S)
&=\sum_{x^{k_2}y\in S\atop 0\le k_2\le 2m-1}\omega^{jk_2}
=\left(1+(-1)^j\right)\sum_{x^{k_2}y\in S\atop 0\le k_2\le m-1}\omega^{jk_2},
\end{align*}
 and, for $1\le i\le 4$, $\lambda_i=\lambda_i(S)$ are respectively given as follows.
\begin{itemize}
\item[$(1)$]
 When $m$ is odd,
\begin{align*}
 \lambda_1
&=2(\sigma^{\mathrm{e}}_{1}+\sigma^{\mathrm{o}}_{1})+\delta +2(\sigma^{\mathrm{e}}_{2}+\sigma^{\mathrm{o}}_{2}), \\ 
 \lambda_2
&=2(\sigma^{\mathrm{e}}_{1}+\sigma^{\mathrm{o}}_{1})+\delta -2(\sigma^{\mathrm{e}}_{2}+\sigma^{\mathrm{o}}_{2}), \\
 \lambda_3
&=2(\sigma^{\mathrm{e}}_{1}-\sigma^{\mathrm{o}}_{1})-\delta, \\
 \lambda_4
&=2(\sigma^{\mathrm{e}}_{1}-\sigma^{\mathrm{o}}_{1})-\delta.
\end{align*}
\item[$(2)$]
 When $m$ is even,
\begin{align*}
 \lambda_1
&=2(\sigma^{\mathrm{e}}_{1}+\sigma^{\mathrm{o}}_{1})+\delta +2(\sigma^{\mathrm{e}}_{2}+\sigma^{\mathrm{o}}_{2}), \\
 \lambda_2
&=2(\sigma^{\mathrm{e}}_{1}+\sigma^{\mathrm{o}}_{1})+\delta -2(\sigma^{\mathrm{e}}_{2}+\sigma^{\mathrm{o}}_{2}), \\
 \lambda_3
&=2(\sigma^{\mathrm{e}}_{1}-\sigma^{\mathrm{o}}_{1})+\delta +2(\sigma^{\mathrm{e}}_{2}-\sigma^{\mathrm{o}}_{2}), \\
 \lambda_4
&=2(\sigma^{\mathrm{e}}_{1}-\sigma^{\mathrm{o}}_{1})+\delta-2(\sigma^{\mathrm{e}}_{2}-\sigma^{\mathrm{o}}_{2}).
\end{align*}
\end{itemize}
\end{lemma}
\begin{proof}
 These are directly obtained from the above tables of the values of the irreducible representations of $Q_{4m}$.
 We notice that $\lambda_i=M_{\chi_i}(S)$ and 
 $\{\mu_j^{+},\mu_j^{-}\}=\Spec(M_{\varphi_j}(S))=\Spec\left(\left[\begin{array}{cc}z_{j}&w_j\\\overline{w_j}&z_j\end{array}\right]\right)$.
\end{proof}

 Remark that $\lambda_1=|S|$, which corresponds to the trivial character, is the largest eigenvalue of $X(S)$.
 We also notice that $z_j\in\mathbb{R}$ and $w_j=0$ if $j$ is odd.

 The following lemma is useful in the case of estimating the eigenvalues of $X(S)$
 corresponding to $\lambda_i$. 
 
\begin{lemma}
\label{lem:onedimeigen}
 Fix $l\in\cL$ and $(l_1,l_2)\in\cL_l$.
 Let $S\in\cS_{l_1,l_2}$.
 Then, we have $\lambda_1=4m-l$ and $\lambda_2=-l+2l_2$.
 Moreover, 
\begin{itemize}
\item[$(1)$]
 \ when $m$ is odd and
\begin{description}
\item[$(i)$]
 $l$ is odd, we have $-(l-l_2)\le \lambda_3=\lambda_4\le l-l_2-2$.
 The absolute values $|\lambda_3|=|\lambda_4|$ of $\lambda_3$ and $\lambda_4$
 take the maximum value $l-l_2$ if and only if $(\sigma^{e}_1,\sigma^{o}_1)=(\frac{m+1}{2}-\frac{l-l_2+1}{2},\frac{m-1}{2})$.
\item[$(ii)$]
 $l$ is even, we have  $-(l-l_2-2)\le \lambda_3=\lambda_4\le l-l_2-2$.
 The absolute values $|\lambda_3|=|\lambda_4|$ of $\lambda_3$ and $\lambda_4$
 take the maximum value $l-l_2-2$ if and only if $(\sigma^{e}_1,\sigma^{o}_1)=(\frac{m+1}{2}-\frac{l-l_2}{2},\frac{m-1}{2})$ or $(\frac{m-1}{2},\frac{m+1}{2}-\frac{l-l_2}{2})$.
\end{description}
\item[$(2)$]
 \ when $m$ is even, we have
 $-l\le \lambda_3\le l+2\delta-4$ and  $-l\le \lambda_4\le l+2\delta-4$.
 The absolute values $|\lambda_3|$ of $\lambda_3$ takes the maximum value $l$ if and only if
 $(\sigma^{\mathrm{e}}_{1},\sigma^{\mathrm{o}}_{1})=(\frac{m}{2}-\frac{l-l_2+\delta}{2},\frac{m}{2})$ and $(\sigma^{\mathrm{e}}_{2},\sigma^{\mathrm{o}}_{2})=(\frac{m}{2}-\frac{l_2}{2},\frac{m}{2})$.
 Similarly, the absolute values $|\lambda_4|$ of $\lambda_4$ takes the maximum value $l$ if and only if
 $(\sigma^{\mathrm{e}}_{1},\sigma^{\mathrm{o}}_{1})=(\frac{m}{2}-\frac{l-l_2+\delta}{2},\frac{m}{2})$ and $(\sigma^{\mathrm{e}}_{2},\sigma^{\mathrm{o}}_{2})=(\frac{m}{2},\frac{m}{2}-\frac{l_2}{2})$.
\end{itemize}
\end{lemma}
\begin{proof}
 The claims on $\lambda_1$ and $\lambda_2$ are clear from Lemma~\ref{lem:eigenQ4m}
 with the expressions $l_1=2m-(2(\sigma^{\mathrm{e}}_{1}+\sigma^{\mathrm{o}}_{1})+\delta)$, $l_2=2m-2(\sigma^{\mathrm{o}}_{2}+\sigma^{\mathrm{o}}_{2})$ and $l_1=l-l_2$.
 Now, let us consider the other cases.
 
 When $m$ is odd,
 since $0\le \sigma^{\mathrm{e}}_{1}\le \frac{m-1}{2}$ and $0\le \sigma^{\mathrm{o}}_{1}\le \frac{m-1}{2}$,
 we see that $\sigma^{\mathrm{e}}_{1}+\sigma^{\mathrm{o}}_{1}=m-\frac{l-l_2+\delta}{2}$ implies that
 $-\frac{l-l_2+\delta-2}{2}\le \sigma^{\mathrm{e}}_{1}-\sigma^{\mathrm{o}}_{1}\le \frac{l-l_2+\delta-2}{2}$.
 This shows that $-(l-l_2-2)-2\delta\le \lambda_3=\lambda_4\le l-l_2-2$.
 Now $|\lambda_3|=|\lambda_4|$ takes maximum value $l-l_2$ if $l$ is odd,
 which is indeed realized when $-\frac{l-l_2+\delta-2}{2}=\sigma^{\mathrm{e}}_{1}-\sigma^{\mathrm{o}}_{1}$ (with $\delta=1$),
 and $l-l_2-2$ otherwise, which is realized when $\sigma^{\mathrm{e}}_{1}-\sigma^{\mathrm{o}}_{1}=\pm \frac{l-l_2+\delta-2}{2}$ (with $\delta=0$).

 When $m$ is even,
 since $0\le \sigma^{\mathrm{e}}_{1}\le \frac{m}{2}-1$, $0\le \sigma^{\mathrm{o}}_{1}\le \frac{m}{2}$,
 $0\le \sigma^{\mathrm{e}}_{2}\le \frac{m}{2}$ and $0\le \sigma^{\mathrm{o}}_{2}\le \frac{m}{2}$,
 we see that $\sigma^{\mathrm{e}}_{1}+\sigma^{\mathrm{o}}_{1}=m-\frac{l-l_2+\delta}{2}$ and $\sigma^{\mathrm{e}}_{2}+\sigma^{\mathrm{o}}_{2}=m-\frac{l_2}{2}$ imply that
 $-\frac{l-l_2+\delta}{2}\le \sigma^{\mathrm{e}}_{1}-\sigma^{\mathrm{o}}_{1}\le \frac{l-l_2+\delta-4}{2}$ and $-\frac{l_2}{2}\le \sigma^{\mathrm{e}}_{2}-\sigma^{\mathrm{o}}_{2}\le \frac{l_2}{2}$, respectively.
 This shows that $-l\le \lambda_3\le l+2\delta-4$ and $-l\le \lambda_4\le l+2\delta-4$.
 Similarly as the above, $|\lambda_3|$ takes maximum value $l$ if
 $-\frac{l-l_2+\delta}{2}=\sigma^{\mathrm{e}}_{1}-\sigma^{\mathrm{o}}_{1}$ and $-\frac{l_2}{2}=\sigma^{\mathrm{e}}_{2}-\sigma^{\mathrm{o}}_{2}$ 
 and $|\lambda_4|$ takes $l$ if $-\frac{l-l_2+\delta}{2}=\sigma^{\mathrm{e}}_{1}-\sigma^{\mathrm{o}}_{1}$ and $\sigma^{\mathrm{e}}_{2}-\sigma^{\mathrm{o}}_{2}=\frac{l_2}{2}$.
\end{proof}


\section{Main results}


\subsection{Trivial lower bound of $\tilde{l}$}
 
 We first show that a lower bound of $\tilde{l}$ is obtained 
 by using the trivial estimate of the eigenvalues of Cayley graphs.
 
\begin{lemma}
\label{lem:eigenestimate}
 Assume $|S|\ge 2m$.
 Then, for all $\lambda\in \Lambda(S)$ with $|\lambda|\ne|S|$,
 we have $|\lambda|\le l(S)$. 
\end{lemma}
\begin{proof}
 The claim is clear for the cases $\lambda=\lambda_i$ for $2\le i\le 4$.
 Actually, since $\lambda_i=\sum_{s\in S}\chi_i(s)=-\sum_{s\notin S}\chi_i(s)$, by the orthogonality of characters,
 it holds that $|\lambda_i|\le \min\{|S|,l(S)\}=l(S)$.
 We next consider the cases $\lambda=\mu_j^{\pm}$ for $1\le j\le m-1$.
 Let $|\mu_j|=\max\{|\mu^{+}_j|,|\mu^{-}_{j}|\}$.
 As is the case of the dihedral groups \cite{HiranoKatataYamasaki2},
 we see that  
\begin{equation}
\label{eq:twodimeigenabso}
 |\mu_j|=|z_j|+|w_j|.
\end{equation}
 Hence, since
\begin{align*}
 z_{j}
&=\sum_{x^{k_1}\in S\atop 0\le k_1\le 2m-1}\omega^{jk_1}
=-\sum_{x^{k_1}\notin S\atop 0\le k_1\le 2m-1}\omega^{jk_1}, \\
 w_{j}
&=\sum_{x^{k_2}y\in S\atop 0\le k_2\le 2m-1}\omega^{jk_2}
=-\sum_{x^{k_2}y\notin S\atop 0\le k_2\le 2m-1}\omega^{jk_2},
\end{align*}
 we have $|\mu_j|\le \min\{|S_1|,l_1(S)\}+\min\{|S_2|,l_2(S)\}$.
 Now, it is easy to see that the right-hand side of the inequality is bounded above by $l(S)$.
\end{proof}

\begin{proposition}
 Let $l_0=\Gauss{4\sqrt{m}}-2$.
 Then, we have $\tilde{l}\ge l_0$.
\end{proposition}
\begin{proof}
 From Lemma~\ref{lem:eigenestimate}, we see that if $l(S)\le\RB(S)=2\sqrt{|S|-1}=2\sqrt{(4m-l(S))-1}$,
 then $X(S)$ is Ramanujan.
 Now one sees that this is equivalent to $l(S)\le 4\sqrt{m}-2$ and hence obtain the desired result. 
 Remark that $l(S)\le 4\sqrt{m}-2$ implies that $l(S)\le 2m$, that is, $|S|\ge 2m$ for all $m\ge 1$.
\end{proof}

 We call $l_0$ a trivial bound of $\tilde{l}$.
 Using Lemma~\ref{lem:onedimeigen}, 
 we can easily determine the bound $\tilde{l}$ in the case of $\cS=\cS_{Q_{4m}}$.

\begin{theorem}
\label{thm:allcases}
 We have $\tilde{l}=l_0$.
\end{theorem}
\begin{proof}
 Take any $S\in\cS_{l_0+1}$ with $l_2(S)=0$, that is, $S\in\cS_{l_0+1,0}$. 
 Then, from Lemma~\ref{lem:onedimeigen}, 
 we have $|\lambda_2|=l_0+1$ and hence, by the definition of $l_0$, $|\lambda_2|>\RB(S)$.
 This means that $X(S)$ is not Ramanujan.
\end{proof}


\subsection{A modification}

 From Theorem~\ref{thm:allcases},
 in the case of $\cS=\cS_{Q_{4m}}$,
 we may not expect a connection between our problem on Ramanujan graphs and a problem on analytic number theory,
 as our previous studies in the cases of the cyclic and dihedral groups \cite{HiranoKatataYamasaki1,HiranoKatataYamasaki2}.
 So, we next take another set of Cayley subsets of $Q_{4m}$, that is,
\[
 \cS'=\{S\in\cS_{Q_{4m}}\,|\,l_2(S)\ne 0\}.
\]
 Notice that $l_2(S)\ne 0$ is equivalent to $S_2\ne \langle x\rangle y$.
 This means that the setting on $\cS'$ is reasonable in the sense that
 we do not consider the extreme case $S_2=\langle x\rangle y$.
 Furthermore, put $\cL'=\{l(S)\,|\,S\in\cS'\}$ and $\cS'_l=\cS_l\cap\cS'$.
 Now our new purpose is to determine the bound 
\[
 \tilde{l}'
=\max\{l\in\cL'\,|\,\text{$X(S)$ is Ramanujan for all $S\in \cS'_k$ ($1\le k\le l$)}\}.
\]
 It is clear that
\begin{equation}
\label{eq:trvialld}
 \tilde{l}'\ge l_0.
\end{equation}
 Moreover, it holds that 

\begin{theorem}
 When $m$ is even, we have $\tilde{l}'=l_0$.
\end{theorem}
\begin{proof}
 From Lemma~\ref{lem:onedimeigen} (2),
 we can find $S\in\cS'_{l_0+1}$ with $l_2(S)\ne 0$ satisfying $|\lambda_3|=l_0+1>\RB(S)$ (or $|\lambda_4|=l_0+1>\RB(S)$).
 This immediately shows that $X(S)$ is not Ramanujan.  
\end{proof}

 From this theorem, 
 we may assume in what follows that $m$ is odd.
 Remark that, in this case, from Lemma~\ref{lem:onedimeigen} again, 
 we have $|\lambda_i|<l$ for $2\le i \le 4$ for any $l\in\cL'$ and $S\in\cS'_l$.
 

\subsection{An upper bound of $\tilde{l}'$}

 As is the case of $\cS$,
 it is convenient to decompose $\cS'$ as follows;
\begin{equation}
 \cS'
=\bigsqcup_{l\in\cL'}\cS'_l
=\bigsqcup_{l\in\cL'}\bigsqcup_{(l_1,l_2)\in\cL'_{l}}\cS'_{l_1,l_2},
\end{equation} 
 where 
\[
 \cL'_{l}
=\left\{(l_1,l_2)\in\mathbb{Z}^2\,\left|\,
\begin{array}{l}
 0<l_1\le 2m, \ l_1\equiv l \!\!\! \pmod{2}, \\
 0<l_2<2m, \ l_2\equiv 0 \!\!\! \pmod{2},
\end{array}
 \ l_1+l_2=l
\right.
\right\}
\]
 and $\cS'_{l_1,l_2}=\cS_{l_1,l_2}\cap \cS'$ for $l\in\cL'$ and $(l_1,l_2)\in\cL'_l$.

 The aim of this subsection is to show the following result.
 
\begin{proposition}
\label{prop:l0l0+1}
 For $m\ge 65$, we have $\tilde{l}'=l_0$ or $\tilde{l}'=l_0+1$.
\end{proposition}
 
 Let $l\in\cL'$.
 To prove Proposition~\ref{prop:l0l0+1},
 we first construct $S^{(l_1,l_2)}\in\cS'_{l_1,l_2}$ for each $(l_1,l_2)\in\cL'_{l}$
 such that $X(S^{(l_1,l_2)})$ may have the maximal eigenvalue (in the sense of absolute value) among $X(S)$ with $S\in\cS'_{l_1,l_2}$.
 Let $(l_1,l_2)\in\cL'_{l}$.
 We define $S^{(l_1,l_2)}=S^{(l_1)}_1\sqcup S^{(l_2)}_{2}\in \cS'_{l_1,l_2}$ by  
\begin{align*}
 S^{(l_1)}_1
&=\langle x \rangle\setminus \{1,x^{\pm 1},\ldots,x^{\pm \frac{l_1-2+\delta}{2}}\}\cup\{x^{m}\}^{1-\delta}, \\
 S^{(l_2)}_2
&=\langle x \rangle y\setminus \{y,xy,\ldots,x^{\frac{l_2}{2}-1}y,x^{m}y,x^{m+1}y,\ldots,x^{m+\frac{l_2}{2}-1}y\},
\end{align*}
 where $\delta=1$ if $l$ is odd and $0$ otherwise.
 We respectively write $z_j$, $w_j$ and $|\mu_j|$ as $z^{(l_1,l_2)}_j$, $w^{(l_1,l_2)}_j$ and $|\mu^{(l_1,l_2)}_j|$ when $S=S^{(l_1,l_2)}$.
 Recall that $w^{(l_1,l_2)}_j=0$ when $j$ is odd.
 On the other hand when $j$ is even, we have  
\begin{align*}
 w^{(l_1,l_2)}_j
=-2\sum_{k_2=0}^{\frac{l_2}{2}-1}e^{\frac{2\pi ijk_2}{2m}}
=-2e^{\frac{\pi ij(l_2-2)}{4m}}\frac{\sin{\frac{\pi jl_2}{4m}}}{\sin{\frac{\pi j}{2m}}}.
\end{align*}
 Moreover, $z^{(l_1,l_2)}_j$ is calculated as 
\begin{align*}
 z^{(l_1,l_2)}_j
&=-\left(\sum_{k_1=-\frac{l_1-2+\delta}{2}}^{\frac{l_1-2+\delta}{2}}e^{\frac{2\pi ijk_1}{2m}}+(1-\delta)(-1)^j\right)\\
&=-\left(\frac{\sin{\frac{\pi j(l_1-1+\delta)}{2m}}}{\sin{\frac{\pi j}{2m}}}+(1-\delta)(-1)^j\right).
\end{align*}
 Hence we have 
\begin{equation}
\label{eq:extream}
 |\mu^{(l_1,l_2)}_j|
=
\left(
\frac{\sin{\frac{\pi j(l_1-1+\delta)}{2m}}}{\sin{\frac{\pi j}{2m}}}+(1-\delta)(-1)^j\right)
+\delta_j\left(2\frac{\sin{\frac{\pi jl_2}{4m}}}{\sin{\frac{\pi j}{2m}}}\right),
\end{equation}
 where $\delta_j=1$ if $j$ is even and $0$ otherwise.
 We now focus on the case of $j=2$.

\begin{lemma}
\label{lem:maxmu2}
 Let $l\in\cL'$.
 When $l\equiv r$ \!\!\!\! $\pmod{6}$ for $0\le r\le 5$, we have 
\begin{equation}
\label{for:m2max}
 \max\left\{\left.|\mu_2^{(l_1,l_2)}|\,\right|\,(l_1,l_2)\in\mathcal{L}'_{l}\right\}
=|\mu_2^{(\check{l}_1,\check{l}_2)}|,
\end{equation}
 where $(\check{l}_1,\check{l}_2)=(\frac{l+a_r}{3},\frac{2l-a_r}{3})\in\mathcal{L}'_l$ with
\[ 
 a_1=2, \ \ a_3=0, \ \ a_5=-2, \ \ a_0=0, \ \ a_{2}=4, \ \ a_4=2.
\]
\end{lemma}
\begin{proof}
 It holds that   
\[
 \frac{\partial }{\partial l_2}|\mu^{(l_1,l_2)}_2|
=\frac{\frac{2\pi}{m}}{\sin{\frac{\pi}{m}}}\sin{\frac{\pi (2(l-1+\delta)-l_2)}{4m}}\sin{\frac{\pi (2(l-1+\delta)-3l_2)}{4m}}.
\]
 Hence, noting that $l$, $l_1$ and $l_2$ are small enough rather than $m$,
 we see that, as a continuous function of $l_2$,
 $\frac{\partial }{\partial l_2}|\mu^{(l_1,l_2)}_2|=0$ on $[1,l]$
 if and only if $l_2=\frac{2(l-1+\delta)}{3}$, 
 which means that $|\mu^{(l_1,l_2)}_2|$ is monotone increasing on $[1,\frac{2(l-1+\delta)}{3}]$
 and decreasing on $[\frac{2(l-1+\delta)}{3},l]$.
 Let us find $(l_1^{\pm},l_2^{\pm})\in\cL'_{l}$ such that 
 $l_2^{-}$ is the maximum and $l_2^{+}$ the minimum integer satisfying  
 $l^{-}_2\le \frac{2(l-1+\delta)}{3}\le l_2^{+}$
 (notice that $l_1^{\pm}$ are automatically determined from $l_2^{\pm}$ by $l_1^{\pm}+l_2^{\pm}=l$).
 If we write $l=6k+r$, then one sees that these are respectively given as follows:
\[
\begin{array}{c||c|c|c}
 r & 1 & 3 & 5 \\ 
 \hline
 (l^{-}_1,l^{-}_2) & (2k+1,4k) & (2k+1,4k+2) & (2k+3,4k+2) \\[2pt]
 (l^{+}_1,l^{+}_2) & (2k-1,4k+2) & (2k+1,4k+2) & (2k+1,4k+4) 
\end{array}
\]
\[
\begin{array}{c||c|c|c}
 r & 0 & 2 & 4 \\ 
 \hline
 (l^{-}_1,l^{-}_2) & (2k+2,4k-2) & (2k+2,4k) & (2k+2,4k+2) \\[2pt]
 (l^{+}_1,l^{+}_2) & (2k,4k) & (2k,4k+2) & (2k+2,4k+2)
\end{array}
\]
 Now the result follows from the facts  
 $|\mu^{(l_1^{-},l_2^{-})}_2|>|\mu^{(l_1^{+},l_2^{+})}_2|$ for $r=1,2$,
 $|\mu^{(l_1^{-},l_2^{-})}_2|=|\mu^{(l_1^{+},l_2^{+})}_2|$ for $r=3,4$ and
 $|\mu^{(l_1^{-},l_2^{-})}_2|<|\mu^{(l_1^{+},l_2^{+})}_2|$ for $r=5,0$.
 Namely,
 $(\check{l}_1,\check{l}_2)=(l_1^{-},l_2^{-})$ for $r=1,2$,
 $(l_1^{-},l_2^{-})=(l_1^{+},l_2^{+})$ for $r=3,4$ and 
 $(l_1^{+},l_2^{+})$ for $r=5,0$.
\end{proof}

 Using Lemma~\ref{lem:maxmu2}, 
 we give a proof of Proposition~\ref{prop:l0l0+1}.

\medbreak 

\noindent
{\it Proof of Proposition~\ref{prop:l0l0+1}}. \ 
 It is sufficient to show that 
 there exists $S\in\cL'_{l_0+2}$ such that $X(S)$ is not Ramanujan.
 Actually, let $l_0=\Gauss{4\sqrt{m}}-2\equiv r$ \!\!\!\! $\pmod{6}$ for $0\le r\le 5$. 
 Take $S^{(\check{l}_1,\check{l}_2)}\in\mathcal{S}'_{l_0+2}$
 with $(\check{l}_1,\check{l}_2)=(\frac{l_0+2+a_{r+2}}{3},\frac{2(l_0+2)-a_{r+2}}{3})\in\cL'_{l_0+2}$.
 Here the index of $a_r$ is considered modulo $6$.
 Then, noticing that $4\sqrt{m}-1<l_0+2\le 4\sqrt{m}$, we have 
\begin{align*}
&\ \ \ |\mu^{(\check{l}_1,\check{l}_2)}_2|-\RB(S^{(\check{l}_1,\check{l}_2)})\\
&=\frac{\sin{\frac{\pi(\check{l}_1-1+\delta)}{m}}}{\sin{\frac{\pi}{m}}}+(1-\delta)+2\frac{\sin{\frac{\pi \check{l}_2}{2m}}}{\sin{\frac{\pi}{m}}}-2\sqrt{4m-(\check{l}_1+\check{l}_2)-1}\\
&=\frac{\sin{(\frac{\pi}{m}(\frac{l_0+2+a_{r+2}}{3}-1+\delta)})}{\sin{\frac{\pi}{m}}}+(1-\delta)+2\frac{\sin{(\frac{\pi}{2m}\frac{2(l_0+2)-a_{r+2}}{3}})}{\sin{\frac{\pi}{m}}}\\
&\ \ \ -2\sqrt{4m-(l_0+2)-1}\\
&>\frac{\sin{\frac{\pi}{m}(\frac{4\sqrt{m}-1+a_{r+2}}{3}-1+\delta)}}{\sin{\frac{\pi}{m}}}+(1-\delta)+2\frac{\sin{\frac{\pi}{2m}\frac{2(4\sqrt{m}-1)-a_{r+2}}{3}}}{\sin{\frac{\pi}{m}}}\\
&\ \ \ -2\sqrt{4m-(4\sqrt{m}-1)-1}\\
&=1-\frac{64\pi^2-27}{54}m^{-\frac{1}{2}}+O(m^{-1})
\end{align*} 
 as $m\to\infty$.
 This shows that $|\mu^{(\check{l}_1,\check{l}_2)}_2|>\RB(S^{(\check{l}_1,\check{l}_2)})$ for sufficiently large $m$
 and hence concludes that the corresponding Cayley graph $X(S^{(\check{l}_1,\check{l}_2)})$ is not Ramanujan.
 Actually, one can check that the inequality holds for $m\ge 105$.
 Moreover,
 we can  numerically see that $|\mu^{(\check{l}_1,\check{l}_2)}_2|-\RB(S^{(\check{l}_1,\check{l}_2)})>0$ for $65\le m\le 103$
 (however it does not hold when $m=63$).


\subsection{A characterization of exceptional primes}

 From now on, we concentrate on the case where $m=p$ is odd prime
 (we can perform the similar discussion for general $m$ as in \cite{HiranoKatataYamasaki1}, though it may be complicated).
 We know from Proposition~\ref{prop:l0l0+1} that
 it can be written as $\tilde{l}'=l_0+\varepsilon$ for some $\varepsilon=\varepsilon_p\in\{0,1\}$.
 As is the case of the cyclic and dihedral graphs \cite{HiranoKatataYamasaki1,HiranoKatataYamasaki2},
 we call $p$ exceptional if $\varepsilon=1$ and ordinary otherwise.
 Now our task is to clarify which $p\in\mathbb{P}$ is exceptional.

 For $l\in\cL'$,
 let $\lambda(l)=\max\{\lambda(S)\,|\,S\in\cS'_{l}\}$ and $\RB(l)=2\sqrt{4p-l-1}$,
 which is nothing but the Ramanujan bound of $X(S)$ for $S\in \cS'_{l}$.
 From the definition, $p$ is exceptional if and only if $\lambda(l_0+1)\le \RB(l_0+1)$.

\begin{lemma}
\label{lem:mml1l2}
 Let $l\in\cL'$.
 For $(l_1,l_2)\in\cL'_l$,
 let $ \lambda(l_1,l_2)=\max\{\lambda(S)\,|\,S\in\cS'_{l_1,l_2}\}$.
 Then, we have $\lambda(l_1,l_2)=|\mu_2^{(l_1,l_2)}|$ for sufficiently large $p$.
\end{lemma}
\begin{proof}
 Take any $S\in\mathcal{S}'_{l_1,l_2}$.
 When $j$ is odd, since $w_j=0$, we have 
 $|\mu_j|=|z_j|\le |z^{(l_1,l_2)}_1|=|\mu_1^{(l_1,l_2)}|$ because $p$ is prime.
 On the other hand when $j$ is even,
 since $jk$ is always even modulo $2p$ for any $k$,
 it holds that $|\mu_j|\le |\mu_2^{(l_1,l_2)}|$ by the same reason as above.
 Moreover, since   
\begin{align*}
 |\mu_1^{(l_1,l_2)}|
&=\frac{\sin{\frac{\pi (l_1-1+\delta)}{2p}}}{\sin{\frac{\pi }{2p}}}-(1-\delta)
=(-2+2\delta+l_1)+O(p^{-2}),\\
 |\mu_2^{(l_1,l_2)}|
&=
\frac{\sin{\frac{\pi (l_1-1+\delta)}{p}}}{\sin{\frac{\pi }{p}}}+(1-\delta)
+2\frac{\sin{\frac{\pi l_2}{2p}}}{\sin{\frac{\pi}{p}}}
=(l_1+l_2)+O(p^{-2}),
\end{align*}
 we see that
 $|\mu_2^{(l_1,l_2)}|-|\mu_1^{(l_1,l_2)}|=l_2+2-2\delta+O(p^{-2})$ as $p\to\infty$.
 Hence, under the condition $l_2>0$,
 we have $|\mu_2^{(l_1,l_2)}|>|\mu_1^{(l_1,l_2)}|$ for sufficiently large $p$.
 Combining this together with the fact $\max\{|\lambda_i|\,|\,2\le i\le 4,\ S\in\cS'_{l_1,l_2}\}=l_1+l_2-2<|\mu_2^{(l_1,l_2)}|$  for sufficiently large $p$,
 one obtains the claim. 
\end{proof}

\begin{proposition}
\label{prop:lotrue}
 Let $p\ge 67$.
 When $l_0\equiv r$ \!\!\!\! $\pmod{6}$ for $0\le r\le 5$,
 we have 
\[
 \lambda(l_0+1)=|\mu_2^{(\check{l}_1,\check{l}_2)}|,
\]
 where $(\check{l}_1,\check{l}_2)=(\frac{l_0+1+a_{r+1}}{3},\frac{2(l_0+1)-a_{r+1}}{3})\in\mathcal{L}'_{l_0+1}$.
\end{proposition}
\begin{proof}
 This follows immediately from Lemma~\ref{lem:maxmu2} and \ref{lem:mml1l2}.
 Remark that the inequality $|\mu_2^{(\check{l}_1,\check{l}_2)}|-|\mu_1^{(\check{l}_1,\check{l}_2)}|>0$ in fact holds for $p\ge 67$.
\end{proof}

 Write $l_0=\Gauss{4\sqrt{p}}-2$ as     
\[
 l_0=24k+r
\]
 for $k\ge 0$ and $0\le r\le 23$. 
 In this case, we see that $p\in I_{r,k}\cap \mathbb{P}$ where 
\begin{align*}
 I_{r,k}
&=\bigl\{t\in\mathbb{R}\,\bigr|\,\Gauss{4\sqrt{t}}-2=24k+r\bigr\}\\
&=\left[36k^2+3(r+2)k+\frac{(r+2)^2}{16},36k^2+3(r+3)k+\frac{(r+3)^2}{16}\right).
\end{align*}
 In other words,
 $p$ can be written as $p=f_{r,c}(k)$ for some integers $k\ge 0$ and $c\in\mathbb{Z}$ 
 with $f_{r,c}(x)$ being a quadratic polynomial defined by 
\[
 f_{r,c}(x)=36x^2+3(r+3)x+c
\]
 and $-3k+\lceil{\frac{(r+2)^2}{16} \rceil}\le c\le \Gauss{\frac{(r+3)^2}{16}}$.

 For $0\le r\le 23$,
 let $I_{r}=\bigsqcup_{k\ge 0}I_{r,k}\cap\mathbb{P}$ and $C_r=\{\Gauss{\frac{(r+3)^2}{16}}+s\,|\,-5\le s\le 0\}$.
 Moreover, let $C'_r=\{c\in C_r\,|\,\text{$f_{r,c}(x)$ is irreducible over $\mathbb{Z}$}\}$.
 Furthermore, for $c\in C'_r$, define $k_{r,c}\in\mathbb{Z}$ as in Table~2. 
 The following is our main result,
 which gives a characterization for the exceptional primes.

\begin{theorem}
\label{thm:mainD}
 A prime $p\in I_r$ with $p\ge 67$ is exceptional if and only if
 it is of the form of $p=f_{r,c}(k)$ for some $c\in C'_r$ and $k\ge k_{r,c}$.
\end{theorem}
\begin{proof}
 We first notice that,
 from the previous discussion with Proposition~\ref{prop:lotrue},
 $p$ is exceptional if and only if $|\mu_{2}(\check{l}_1,\check{l}_2)|\le \RB(l_0+1)$.
 To clarify when this inequality holds, 
 we introduce an interpolation function $F_r(t)$ of the difference between  
 $|\mu_{2}(\check{l}_1,\check{l}_2)|$ and $\RB(l_0+1)$ on $I_{r,k}$, that is,
\begin{align*}
 F_{r}(t)
&=\frac{\sin{\frac{\pi(8k+\frac{r+1+a_{r+1}}{3}-1+\delta)}{t}}}{\sin{\frac{\pi}{t}}}+(1-\delta)
+2\frac{\sin{\frac{\pi(16k+\frac{2(r+1)-a_{r+1}}{3})}{2t}}}{\sin{\frac{\pi}{t}}}\\
&\ \ \ -2\sqrt{4t-(24k+r+1)-1}.
\end{align*}
 Notice that 
 $(\check{l}_1,\check{l}_2)=\bigl(8k+\frac{r+1+a_{r+1}}{3},16k+\frac{2(r+1)-a_{r+1}}{3}\bigr)$
 when $l_0=24k+r$.
 One can see that $F_{r}(t)$ is monotone decreasing on $I_{r,k}$ for sufficiently large $k$.
 Moreover, at $t=p=f_{r,c}(k)\in I_{r,k}\cap\mathbb{P}$, one has
\[
 F_{r}(p)=\frac{27(r+3)^2-432c-256\pi^2}{1296}k^{-1}+O(k^{-2})
\]
 as $k\to\infty$ because 
\begin{align*}
 |\mu_{2}(\check{l_1},\check{l_2})|
&=\frac{\sin{\frac{\pi(8k+\frac{r+1+a_{r+1}}{3}-1+\delta)}{36k^2+3(r+3)k+c}}}{\sin{\frac{\pi}{36k^2+3(r+3)k+c}}}+(1-\delta)
+2\frac{\sin{\frac{\pi(16k+\frac{2(r+1)-a_{r+1}}{3})}{2(36k^2+3(r+3)k+c)}}}{\sin{\frac{\pi}{36k^2+3(r+3)k+c}}}\\
&=24k+(1+r)-\frac{16 \pi^2}{81}k^{-1}+O(k^{-2}),\\
 \mathrm{RB}(l_0+1)
&=2\sqrt{4(36k^2+3(r+3)k+c)-(24k+r+1)-1}\\
&=24 k+(1+r)-\frac{(r+3)^2-16c}{48}k^{-1}+O(k^{-2}).
\end{align*}
 This shows that $F_{r}(p)<0$ for sufficiently large $k$ if and only if
 $27(r+3)^2-432c-256\pi^2<0$, in other words, $\lceil{\frac{27(r+3)^2-256\pi^2}{432} \rceil} \le c$.
 Here, we see that $\lceil{\frac{27(r+3)^2-256\pi^2}{432} \rceil}=\Gauss{\frac{(r+3)^2}{16}}-5$ for all $0\le r\le 23$,
 which means that $c\in C_r$.
 Moreover, since $f_{r,c}(k)$ does not express any prime if $f_{r,c}(x)$ is not irreducible over $\mathbb{Z}$, $c$ must be in $C'_r$. 
 Furthermore, it is checked that, for each $0\le r\le 23$ and $c\in C'_r$,
 the inequalities $f_{r,c}(k)\ge 67$ and $F_{r}(p)<0$ for $p=f_{r,c}(k)$ hold if and only if $k\ge k_{r,c}$.
 This completes the proof of the theorem.
\end{proof}

\begin{table}[htbp]
\begin{center}
{\footnotesize
\begin{tabular}{c|c||c|c|c|c|c}
 $r$ & $c\in C'_r$ & $f_{r,c}(x)$ & $k_{r,c}$ & $J_{r,c}$ & $N_{r,c}$ & $\frac{C(f_{r,c})}{2\delta_r}$ \\
\hline
\hline
 $0$ & $-5$ & $36x^2+9x-5$ & $9$ & $7177, 11821, 20947, 52321, 121621$ & $9597$ & $0.24501$ \\
\hline
 $0$ & $-4$ & $36x^2+9x-4$ & $2$ & $347, 941, 1823, 4451, 6197$ & $17722$ & $0.45086$ \\
\hline
 $0$ & $-2$ & $36x^2+9x-2$ & $2$ & $349, 6199, 8233, 16063, 19249$ & $11061$ & $0.28123$ \\
\hline
\hline
 $1$ & $-1$ & $36x^2+12x-1$ & $2$ & $167, 359, 1367, 1847, 2399$ & $24414$ & $0.61666$ \\
\hline
\hline
 $2$ & $-4$ & $36x^2+15x-4$ & $9$ & $4517, 16187, 22871, 30707, 44621$ & $9685$ & $0.24501$ \\
\hline
 $2$ & $-2$ & $36x^2+15x-2$ & $1$ & $367, 1867, 3049, 4519, 6277$ & $13501$ & $0.34106$ \\
\hline
 $2$ & $-1$ & $36x^2+15x-1$ & $1$ & $173, 2423, 11933, 14699, 28643$ & $11181$ & $0.28123$ \\
\hline
\hline
 $3$ & $-1$ & $36x^2+18x-1$ & $2$ & $179, 647, 1889, 2447, 3779$ & $31692$ & $0.80725$ \\
\hline
 $3$ & $1$ & $36x^2+18x+1$ & $1$ & $181, 379, 991, 3079, 7309$ & $23288$ & $0.59109$ \\
\hline
\hline
 $4$ & $-1$ & $36x^2+21x-1$ & $1$ & $659, 7349, 9551, 12041, 33029$ & $10633$ & $0.26894$ \\
\hline
 $4$ & $1$ & $36x^2+21x+1$ & $1$ & $661, 1423, 2473, 5437, 7351$ & $15712$ & $0.40086$ \\
\hline
 $4$ & $2$ & $36x^2+21x+2$ & $2$ & $389, 1913, 6359, 13397, 16319$ & $15405$ & $0.39341$ \\
\hline
\hline
 $5$ & $-1$ & $36x^2+24x-1$ & $2$ & $191, 1019, 1439, 1931, 5471$ & $23332$ & $0.59109$ \\
\hline
 $5$ & $1$ & $36x^2+24x+1$ & $2$ & $193, 397, 673, 1021, 1933$ & $27255$ & $0.69166$ \\
\hline
\hline
 $6$ & $1$ & $36x^2+27x+1$ & $1$ & $199, 1459, 2521, 9649, 33211$ & $10609$ & $0.26894$ \\
\hline
 $6$ & $4$ & $36x^2+27x+4$ & $1$ & $67, 409, 1039, 3163, 4657$ & $15494$ & $0.39341$ \\
\hline
\hline
 $7$ & $1$ & $36x^2+30x+1$ & $2$ & $1051, 3187, 7477, 9697, 13567$ & $18210$ & $0.46393$ \\
\hline
 $7$ & $5$ & $36x^2+30x+5$ & $1$ & $71, 419, 701, 1481, 1979$ & $23192$ & $0.59109$ \\
\hline
\hline
 $8$ & $2$ & $36x^2+33x+2$ & $9$ & $4721, 8597, 23327, 61871, 81077$ & $9591$ & $0.24501$ \\
\hline
 $8$ & $4$ & $36x^2+33x+4$ & $1$ & $73, 1069, 1999, 3217, 4723$ & $13526$ & $0.34106$ \\
\hline
 $8$ & $5$ & $36x^2+33x+5$ & $1$ & $1499, 7523, 9749, 12263, 29153$ & $10933$ & $0.28123$ \\
\hline
\hline
 $9$ & $7$ & $36x^2+36x+7$ & $1$ & $79, 223, 439, 727, 1087$ & $24281$ & $0.61666$ \\
\hline
\hline
 $10$ & $5$ & $36x^2+39x+5$ & $9$ & $5657, 7607, 18287, 65147, 99377$ & $9537$ & $0.24501$ \\
\hline
 $10$ & $7$ & $36x^2+39x+7$ & $1$ & $229, 739, 5659, 12373, 15187$ & $13322$ & $0.34106$ \\
\hline
 $10$ & $8$ & $36x^2+39x+8$ & $1$ & $83, 449, 1103, 4793, 6599$ & $11175$ & $0.28123$ \\
\hline
\hline
 $11$ & $7$ & $36x^2+42x+7$ & $2$ & $457, 751, 1117, 2647, 3301$ & $18110$ & $0.46393$ \\
\hline
 $11$ & $11$ & $36x^2+42x+11$ & $1$ & $89, 239, 461, 1559, 2069$ & $23297$ & $0.59109$ \\
\hline
\hline
 $12$ & $10$ & $36x^2+45x+10$ & $1$ & $2089, 3331, 4861, 6679, 16831$ & $10588$ & $0.26894$ \\
\hline
 $12$ & $13$ & $36x^2+45x+13$ & $1$ & $769, 1579, 2677, 5737, 7699$ & $15505$ & $0.39341$ \\
\hline
\hline
 $13$ & $11$ & $36x^2+48x+11$ & $1$ & $251, 479, 1151, 2111, 2699$ & $23137$ & $0.59109$ \\
\hline
 $13$ & $13$ & $36x^2+48x+13$ & $1$ & $97, 1153, 1597, 2113, 3361$ & $27257$ & $0.69166$ \\
\hline
\hline
 $14$ & $14$ & $36x^2+51x+14$ & $1$ & $101, 491, 3389, 4931, 6761$ & $10559$ & $0.26894$ \\
\hline
 $14$ & $16$ & $36x^2+51x+16$ & $1$ & $103, 1171, 2137, 3391, 4933$ & $15790$ & $0.40086$ \\
\hline
 $14$ & $17$ & $36x^2+51x+17$ & $1$ & $263, 797, 1619, 2729, 4127$ & $15393$ & $0.39341$ \\
\hline
\hline
 $15$ & $17$ & $36x^2+54x+17$ & $1$ & $107, 269, 503, 809, 1187$ & $31685$ & $0.80725$ \\
\hline
 $15$ & $19$ & $36x^2+54x+19$ & $1$ & $109, 271, 811, 2161, 4159$ & $23208$ & $0.59109$ \\
\hline
\hline
 $16$ & $17$ & $36x^2+57x+17$ & $8$ & $2777, 29837, 34127, 54167, 72221$ & $9606$ & $0.24501$ \\
\hline
 $16$ & $19$ & $36x^2+57x+19$ & $1$ & $277, 823, 1657, 7873, 15559$ & $13448$ & $0.34106$ \\
\hline
 $16$ & $20$ & $36x^2+57x+20$ & $1$ & $113, 3449, 5003, 11393, 17093$ & $11096$ & $0.28123$ \\
\hline
\hline
 $17$ & $23$ & $36x^2+60x+23$ & $1$ & $839, 1223, 2207, 5039, 5927$ & $24229$ & $0.61666$ \\
\hline
\hline
 $18$ & $22$ & $36x^2+63x+22$ & $8$ & $9067, 11497, 24097, 27967, 36571$ & $9662$ & $0.24501$ \\
\hline
 $18$ & $23$ & $36x^2+63x+23$ & $1$ & $293, 1697, 4253, 10247, 12821$ & $17614$ & $0.45086$ \\
\hline
 $18$ & $25$ & $36x^2+63x+25$ & $1$ & $853, 1699, 2833, 7963, 12823$ & $10918$ & $0.28123$ \\
\hline
\hline
 $19$ & $25$ & $36x^2+66x+25$ & $1$ & $127, 547, 2251, 2857, 5107$ & $18271$ & $0.46393$ \\
\hline
 $19$ & $29$ & $36x^2+66x+29$ & $1$ & $131, 1259, 1721, 2861, 3539$ & $23270$ & $0.59109$ \\
\hline
\hline
 $20$ & $29$ & $36x^2+69x+29$ & $1$ & $311, 881, 15809, 34499, 43991$ & $10567$ & $0.26894$ \\
\hline
 $20$ & $31$ & $36x^2+69x+31$ & $1$ & $313, 883, 1741, 2887, 6043$ & $15875$ & $0.40086$ \\
\hline
 $20$ & $32$ & $36x^2+69x+32$ & $1$ & $137, 563, 1277, 5147, 7013$ & $15649$ & $0.39341$ \\
\hline
\hline
 $21$ & $31$ & $36x^2+72x+31$ & $1$ & $139, 571, 1291, 1759, 5179$ & $23262$ & $0.59109$ \\
\hline
\hline
 $22$ & $35$ & $36x^2+75x+35$ & $1$ & $911, 2939, 13049, 22571, 26321$ & $10591$ & $0.26894$ \\
\hline
 $22$ & $37$ & $36x^2+75x+37$ & $1$ & $331, 1783, 6121, 10453, 15937$ & $15764$ & $0.40086$ \\
\hline
 $22$ & $38$ & $36x^2+75x+38$ & $1$ & $149, 587, 11717, 17489, 20807$ & $15460$ & $0.39341$ \\
\hline
\hline
 $23$ & $37$ & $36x^2+78x+37$ & $2$ & $337, 1327, 1801, 2347, 10501$ & $18177$ & $0.46393$ \\
\hline
 $23$ & $41$ & $36x^2+78x+41$ & $1$ & $599, 929, 2351, 2969, 3659$ & $23223$ & $0.59109$ 
\end{tabular}
}
\end{center}
\caption{The fifty-four quadratic polynomials $f_{r,c}(x)$ for $0\le r\le 23$ and $c\in C'_r$.}
\end{table}

 For $0\le r\le 23$ and $c\in C'_r$,
 let
\[
 J_{r,c}=\{p\,|\,\text{$p=f_{r,c}(k)\in I_r$ for some $k\ge k_{r,c}$}\}.
\]
 Namely, $J_{r,c}$ is the set of exceptional primes $p$ of the form of $p=f_{r,c}(k)$.
 We show the first five such primes in Table~2 for each $r$ and $c$.

 The classical Hardy-Littlewood conjecture \cite[Conjecture~F]{HardyLittlewood1923} asserts that
 if a quadratic polynomial $f(x)=ax^2+bx+c$ with $a,b,c\in\mathbb{Z}$ satisfies the conditions that
 $a>0$, $a,b$ and $c$ are relatively prime, $a+b$ and $c$ are not both even
 and the discriminant $D(f)=b^2-4ac$ of $f$ is not a square,
 then there are infinitely many primes represented by $f(x)$ and, 
 moreover, that
\[
 \pi(f;x)
=\#\{p\le x\,|\,\text{$p=f(k)\in\mathbb{P}$ for some $k\ge 0$}\}
\]
 obeys the asymptotic behavior 
\[
 \pi(f;x)\sim \frac{\varepsilon(f)C(f)}{\sqrt{a}}\prod_{p\mid a,\ p\mid b \atop p\ge 3}\frac{p}{p-1}\cdot \frac{\sqrt{x}}{\log{x}}
\]
 as $x\to\infty$ where $\varepsilon(f)$ is $1$ if $a+b$ is odd and $2$ otherwise and 
\[
 C(f)=\prod_{p\,\nmid\,a \atop p\ge 3}\Bigl(1-\frac{\bigl(\frac{D(f)}{p}\bigr)}{p-1}\Bigr)
\]
 with $\bigl(\frac{D}{p})$ being the Legendre symbol.
 The constant $C(f)$ is called the Hardy-Littlewood constant of $f$.
 Because the polynomial $f_{r,c}(x)$ satisfies the above conditions, 
 we can expect that it indeed represents infinitely many primes.
 In our case, it may hold that  
\[
 \pi(f_{r,c};x)\sim \frac{C(f_{r,c})}{2\delta_r}\frac{\sqrt{x}}{\log{x}}, \quad
 C(f_{r,c})=\prod_{p\ge 5}\Bigl(1-\frac{\bigl(\frac{(r+3)^2-16c}{p}\bigr)}{p-1}\Bigr).
\]
 We also show both the numerical value of $\frac{C(f_{r,c})}{2\delta_r}$
 and the exact number of $N_{r,c}=\#\{p\le 10^{12}\,|\,\text{$p=f_{r,c}(k)\in I_r$ for some $k\ge k_{r,c}$}\}$ in Table~2.
 Notice that $\frac{\sqrt{x}}{\log{x}}=36191.20\ldots$ when $x=10^{12}$.

 The following is immediate from Theorem~\ref{thm:mainD}.
 
\begin{corollary}
 There exists infinitely many exceptional primes if  
 the Hardy-Littlewood conjecture is true for at least one of $f_{r,c}(x)$ for $0\le r\le 23$ and $c\in C'_r$.
\end{corollary}

 We notice that, if we can show that
 there exists infinitely many exceptional primes (in the frame work of the graph theory) on the other hand,
 then it implies that at least one of $f_{r,c}(t)$ represents infinitely many primes.

 We also remark that though we omit to show here but
 the existence of infinitely many ordinary primes is similarly verified
 by using Dirichlet's theorem on arithmetic progressions as \cite{HiranoKatataYamasaki1,HiranoKatataYamasaki2}.



\bigskip 

\noindent
\textsc{Yoshinori YAMASAKI}\\
 Graduate School of Science and Engineering, Ehime University,\\
 Bunkyo-cho, Matsuyama, 790-8577 JAPAN.\\
 \texttt{yamasaki@math.sci.ehime-u.ac.jp}

\end{document}